\documentclass{amsart}
\usepackage{setspace}
\usepackage{graphicx}
\usepackage{amsmath}
\usepackage{amscd}
\usepackage{amsfonts}
\usepackage{amssymb}
\usepackage{amsthm}
\usepackage{mathptmx}
\usepackage{mathrsfs}
\usepackage{latexsym}
\usepackage{times}
\usepackage[mathscr]{euscript}
\usepackage[isolatin]{inputenc}

\DeclareMathAlphabet{\mathpzc}{OT1}{pzc}{m}{it}

\newtheorem{thm}{Theorem}[section]
\newtheorem{lem}[thm]{Lemma}
\newtheorem{prop}[thm]{Proposition}
\newtheorem{cor}[thm]{Corollary}

\theoremstyle{definition}
\newtheorem{defn}[thm]{Definition}
\newtheorem{ex}[thm]{Example}

\theoremstyle{remark}

\newcommand{\ok}{\mathfrak{o}}

\newcommand{\zk}{\mathfrak{z}}

\newcommand{\g}{\mathfrak{g}}

\newcommand{\hk}{\mathfrak{h}}
\newcommand{\fk}{\mathfrak{f}}

\newcommand{\qk}{\mathfrak{q}}

\newcommand{\jk}{\mathfrak{j}}

\newcommand\CC{\mathbb C}
\newcommand\NN{\mathbb N}

\newcommand{\io}{\operatorname{\iota}}

\newcommand{\ad}{\operatorname{ad}}

\newcommand{\Der}{\operatorname{Der}}

\renewcommand\hat\widehat
\renewcommand\tilde\widetilde 
\newcommand{\spa}{\operatorname{span}}

\newcommand\Wedge{\bigwedge}

\begin{document}
\title[The Betti numbers for a family of solvable Lie algebras]{The Betti numbers for a family of solvable Lie algebras}

\author{Thanh Minh Duong}

\address{Department of Physics, Ho Chi Minh city University of Pedagogy, 280 An Duong Vuong, Ho Chi Minh city, Vietnam.}
\email{thanhdmi@hcmup.edu.vn}
\footnotetext{The author was supported by the Foundation for Science and Technology Project of Vietnam Ministry of Education and Training.}
\begin{abstract} 
  We give a characterization of symplectic quadratic Lie algebras that their Lie algebra of inner derivations has an invertible derivation. A family of symplectic quadratic Lie algebras is introduced to illustrate this situation. Finally, we calculate explicitly the Betti numbers of a family of solvable Lie algebras in two ways: using the cohomology of quadratic Lie algebras and applying a Pouseele's result on extensions of the one-dimensional Lie algebra by Heisenberg Lie algebras.
\end{abstract}
\maketitle
\section{Introduction}

Let $\g$ be a complex Lie algebra endowed with a non-degenerate invariant symmetric bilinear form $B$, $\{X_1,...,X_n\}$ be a basis of $\g$ and $\{\omega_1,...,\omega_n\}$ be its dual basis. Denote by $\{Y_1,...,Y_n\}$ the basis of $\g$ defined by $B(Y_i,.)=\omega_i$, $1\leq i\leq n$.  Pinczon and Ushirobira discovered in \cite{PU07} that the differential $\partial$ on $\Wedge(\g^*)$, the space of antisymmetric forms on $\g$, is given by $\partial:=-\{I,.\}$ where $I$ is defined by:
\[
I(X,Y,Z)=B([X,Y],Z),\ \ \forall\ X,\ Y,\ Z\in\g
\]
and $\{~,~\}$ is the super Poisson bracket on $\Wedge(\g^*)$ defined by
\[\{\Omega, \Omega '\} = (-1)^{k+1} \sum_{i,j} B(Y_i,Y_j)\io_{X_i}(\Omega)
\wedge \io_{X_j}(\Omega'), \ \forall \ \Omega \in \Wedge{}^k(\g^*),\ \Omega'
\in \Wedge(\g^*).\]

In Section 1, by using this, we detail a result of Medina and Revoy in \cite{MR85} that there is an isomorphism between the second cohomology group $H^2(\g,\CC)$ and $\Der_a(\g)/\ad(\g)$ where $\Der_a(\g)$ is the vector space of skew-symmetric derivations of $\g$ and $\ad(\g)$ is its subspace of inner ones. 

Involving in the well-known theorem by Jacobson on the invertibility of Lie algebra derivations that a Lie algebra over a field of characteristic zero is nilpotent
if it admits an invertible derivation, we are interested in Lie algebras having an invertible derivation. We prove that the Lie algebra $\ad(\g)$ of a symplectic quadratic Lie algebra has that property. In particular, we have the following (Proposition \ref{prop1.5}).

\bigskip
{\sc{Theorem 1.}} {\em Let $(\g,B,\omega)$ be a symplectic quadratic Lie algebra. Consider the mapping $\mathcal{D}:\ad(\g)\rightarrow \ad(\g)$ defined by $\mathcal{D}(\ad(X))=\ad\left(\phi^{-1}\left(\iota_X(\omega)\right)\right)$ with $\phi:\g\rightarrow \g^*$, $\phi(X)=B(X,.)$, then $\mathcal{D}$ is an invertible derivation of $\ad(\g)$.}
\bigskip

The reader is referred to \cite{BBM07} for futher information about symplectic quadratic Lie algebras. A family of such algebras is given to illustrate this situation. 

In Section 2, motivated by Corollary 4.4 in \cite{MR85}, we give the Betti numbers for a family of solvable quadratic Lie algebras defined as follows. For each $n\in\NN$, let $\g_{2n+2}$ denote the Lie algebra with basis $\{X_0,...,X_n,Y_0,...,Y_n\}$ and non-zero Lie brackets $[Y_0,X_i]=X_i$, $[Y_0,Y_i]=-Y_i$, $[X_i,Y_i]=X_0$, $1\leq i\leq n$. Denote by $B^k(\g_{2n+2}) = B^k(\g_{2n+2},\CC)$, $Z^k(\g_{2n+2}) = Z^k(\g_{2n+2},\CC)$, $H^k(\g_{2n+2}) = H^k(\g_{2n+2},\CC)$ and $b_{k}= b_{k}(\g_{2n+2},\CC)$. By computing on super Poisson brackets, our second result is the following.

\bigskip
{\sc{Theorem 2.}} {\em The $k^{th}$ Betti numbers of $\g_{2n+2}$ are given as follows:
	
	\begin{enumerate}
		\item If $k$ is even then one has \[b_{k}= \left|\begin{pmatrix} n \\ \frac{k}{2}\end{pmatrix}\begin{pmatrix} n \\ \frac{k}{2}\end{pmatrix} 
			-\begin{pmatrix} n \\ \frac{k-2}{2}\end{pmatrix}\begin{pmatrix} n \\ \frac{k-2}{2}\end{pmatrix}\right|.\]
			\item If $k$ is odd then one has
				
		\begin{itemize}
			\item if $k<n+1$ then \[b_{k}= \begin{pmatrix} n \\ \frac{k-1}{2}\end{pmatrix}\begin{pmatrix} n \\ \frac{k-1}{2}\end{pmatrix} -\begin{pmatrix} n \\ \frac{k-3}{2}\end{pmatrix}\begin{pmatrix} n \\ \frac{k-3}{2}\end{pmatrix},\]
			\item if $k=n+1$ then \[b_{n+1}= 2\begin{pmatrix} n \\ \frac{n}{2}\end{pmatrix}\begin{pmatrix} n \\ \frac{n}{2}\end{pmatrix} -2\begin{pmatrix} n \\ \frac{n+2}{2}\end{pmatrix}\begin{pmatrix} n \\ \frac{n+2}{2}\end{pmatrix},\]
			\item if $k>n+1$ then 
			\[b_{k}= \begin{pmatrix} n \\ \frac{k-1}{2}\end{pmatrix}\begin{pmatrix} n \\ \frac{k-1}{2}\end{pmatrix} - \begin{pmatrix} n \\ \frac{k+1}{2}\end{pmatrix}\begin{pmatrix} n \\ \frac{k+1}{2}\end{pmatrix}.\]
			\end{itemize}
				
 \end{enumerate}} \bigskip

Our method is direct and different from the Pouseele's method given in \cite{Pou05} that we shall recall in Appendix 1. In the Pouseele's method, the Betti numbers of $\g_{2n+2}$ follow the Betti numbers of the $2n+1$-dimensional Lie algebra $\fk$ defined by $[x,x_i]=x_i$ and $[y,y_i]=-y_i$ for all $1\leq i\leq n$.

Other results of Betti numbers for some families of nilpotent Lie algebras, we refer the reader to \cite{ACJ97}, \cite{Pou05} or \cite{San83}.

\section{A characterization of symplectic quadratic Lie algebras}

Let $\g$ be a complex Lie algebra endowed with a non-degenerate invariant symmetric bilinear form $B$. In this case, we call the pair $(\g,B)$ a {\em quadratic} Lie algebra. Denote by $\Der_a(\g)$ the vector space of skew-symmetric derivations of $\g$, that is the vector space of derivations $D$ satisfying $B(D(X),Y) = -B(X,D(Y))$ for all $X,\ Y\in\g$, then $\Der_a(\g)$ is a Lie subalgebra of $\Der(\g)$.

\begin{prop}
There exists a Lie algebra isomorphism $T$ between $\Der_a(\g)$ and the space $\{\Omega\in\Wedge{}^2(\g^*)\ |\ \{I,\Omega\}=0\}$. This isomorphism induces an isomorphism from $\ad(\g)$ onto $\iota_\g(I)=\{\iota_X(I)\in\Wedge{}^2(\g^*)\ |\ X\in\g\}$.
\end{prop}
\begin{proof}
Let $D\in \Der_a(\g)$ and set $\Omega\in\Wedge{}^2(\g^*)$ by $\Omega(X,Y)=B(D(X),Y)$ for all $X,\ Y\in\g$. Then $D$ is a derivation of $\g$ if and only if
\[
\Omega([X,Y],Z)+\Omega([Y,Z],X)+\Omega([Z,X],Y)=0
\]
for all $X,\ Y,\ Z\in\g$. It means $\{I,\Omega\}=0$. Define the map $T$ from $\Der_a(\g)$ onto  $\{\Omega\in\Wedge{}^2(\g^*)\ |\ \{I,\Omega\}=0\}$ by $T(D)=\Omega$ then $T$ is a one-to-one correspondence.

Now we shall show that $T([D,D'])=\{T(D),T(D')\}$ for all $D,\ D'\in \Der_a(\g)$. Indeed, set $\Omega=T(D)$, $\Omega'=T(D')$ and fix an orthonormal basis $\{X_j\}_{j=1}^n$ of $\g$. One has
\[
\{\Omega,\Omega'\}(X,Y)=-\left(\sum_{j=1}^n \io_{X_j}(\Omega)
\wedge \io_{X_j}(\Omega')\right) (X,Y)
\]
\[
=-\sum_{j=1}^n \left(\Omega(X_j,X)\Omega'(X_j,Y)-\Omega(X_j,Y)
\Omega'(X_j,X)\right)
\]
\[
=-\sum_{j=1}^n B\left(B(D(X_j),X)D'(X_j)-B(D'(X_j),X)D(X_j),Y)\right)
\]
\[
=-\sum_{j=1}^n B\left(D'(D(X))-D(D'(X)), Y)\right)
= -B([D',D](X), Y).
\]
That means $T([D,D'])=\{T(D),T(D')\}$ and then $T$ is a Lie algebra isomorphism.

If $D=\ad(X_0)$ then $T(D)(X,Y)=B([X_0,X],Y)=I(X_0,Y,Z)=\iota_{X_0}(I)(X,Y)$. Therefore, $T(D)=\iota_{X_0}(I)$.
\end{proof}
\begin{cor} $\{\iota_X(I),\iota_Y(I)\} =\iota_{[X,Y]}(I)$.
\end{cor}

\begin{cor} \cite{MR85}\hfill

The cohomology group $H^2(\g,\CC)\simeq \Der_a(\g,B)/\ad(\g)$.
\end{cor}
\begin{defn}
A non-degenerate skew-symmetric bilinear form $\omega:\g\times\g\rightarrow\CC$ is called a {\em symplectic structure} on $\g$ if it satisfies
\[
\omega([X,Y],Z)+\omega([Y,Z],X)+\omega([Z,X],Y)=0
\]
for all $X,\ Y,\ Z\in\g$.
\end{defn}

A symplectic structure $\omega$ on a quadratic Lie algebra $(\g,B)$ is corresponding to a skew-symmetric invertible derivation $D$ defined by $\omega(X,Y)=B(D(X),Y)$, for all $X,\ Y\in\g$. As above, a symplectic structure is exactly a non-degenerate 2-form $\omega$ satisfying $\{I,\omega\}=0$. If $\g$ has a such $\omega$  then we call $(\g,B,\omega)$ a {\em symplectic} quadratic Lie algebra.

For symplectic quadratic Lie algebras, the reader can refer to \cite{BBM07} for more details. Here we give a following property.
\begin{prop}\label{prop1.5}
 Let $(\g,B,\omega)$ be a symplectic quadratic Lie algebra. Consider the mapping $\mathcal{D}:\ad(\g)\rightarrow \ad(\g)$ defined by $\mathcal{D}(\ad(X))=\ad\left(\phi^{-1}\left(\iota_X(\omega)\right)\right)$ with $\phi:\g\rightarrow \g^*$, $\phi(X)=B(X,.)$, then $\mathcal{D}$ is an invertible derivation of $\ad(\g)$.
\end{prop}
\begin{proof}

As above we have $\{I,\omega\}=0$ and then $\iota_X\left(\{I,\omega\}\right)=0$ for all $X\in\g$. It implies $\{\iota_X(I),\omega\}=\{I,\iota_X(\omega)\}$ for all $X\in\g$. Note that if $X$ is nonzero, since $\omega$ is non-degenerate then $\iota_X(\omega)$ is non trivial. Set $Y=\phi^{-1}\left(\iota_X(\omega)\right)$ then $\{I,\iota_X(\omega)\} = \iota_Y(I)$ and therefore this defines an inner derivation. Let $D$ be the derivation corresponding to $\omega$ then one has $[\ad(X),D]=\ad(Y)$.

Let $\ad(X)\in \ad(\g)$. Set $\alpha=\phi(X)$. Since $\omega$ is non-degenerate then there exists an element $Y\in\g$ such that $\alpha=\iota_Y(\omega)$. In this case,  $\mathcal{D}(\ad(Y))=\ad(X)$. That means $\mathcal{D}$ onto and therefore it is bijective.
\end{proof}

Next, we give a family of symplectic quadratic Lie algebras that has been defined in \cite{DPU12} as follows. 

\begin{ex}  Let $p \in \NN \setminus \{ 0 \}$. We denote the {\em Jordan block
    of size} $p$ by $J_1 := (0)$ and for $p \geq 2$, \[J_p :
  = \begin{pmatrix} 0 & 1 & 0 & \dots & 0 \\ 0 & 0 & 1 & \dots & 0\\
    \vdots & \vdots & \dots & \ddots & \vdots \\ 0 & 0 & \dots & 0 & 1
    \\ 0 & 0 & 0 & \dots & 0 \end{pmatrix}.\]

For $p \geq 2$, we consider $\qk = \CC^{2p}$ with a basis $\{X_i,\ Y_i\}$, $1\leq i\leq p$, and equipped with a bilinear form $B$ satisfying $B(X_i,X_j)=B(Y_i,Y_j)=0$ and $B(X_i,Y_j)=\delta_{ij}$. Let $C:\qk\rightarrow\qk$ with
  matrix \[ C=\begin{pmatrix} J_p & 0 \\ 0 & -{}^t J_p\end{pmatrix}\] in
  the given basis. Then $C \in \ok(2p)$.
	
	Let $\hk=\CC^2$ and $\{X_0,Y_0\}$ be a basis of $\hk$. Define on the vector space
  $\jk_{2p}=\qk\oplus\hk$ the Lie bracket $[Y_0,X] = C(X)$, $[X,Y] = B(C(X),Y)X_0$ and the bilinear form $\overline{B}(X_0,Y_0)=1$, $\overline{B}(X_0,X_0)=\overline{B}(Y_0,Y_0)=\overline{B}(X_0,X) = \overline{B}(Y_0,X)=0$ and $\overline{B}(X,Y)=B(X,Y)$ for all $X,\ Y\in\qk$. So
  $\jk_{2p}$ is a nilpotent Lie algebra and it will be called a $2p+2$-dimensional {\em nilpotent
  Jordan-type} Lie algebra.
\end{ex}

Denote by $\{\alpha,\ \alpha_1,...,\ \alpha_p,\ \beta,\ \beta_1,...,\ \beta_p\}$ the dual basis of $\{X_0,...,X_p,Y_0,...,Y_p\}$ then $I=\beta\wedge\sum_{i=1}^{p-1}\alpha_{i+1}\wedge\beta_i$. In this case, we choose $\omega = \alpha\wedge \beta+ \sum_{i=1}^{p}i\alpha_{i}\wedge\beta_i$ then $\{I,\omega\} = 0$ and therefore $(\jk_{2p},B,\omega)$ is a symplectic quadratic Lie algebra. Notice that if we define $\mathcal{D}(\ad(Y_0))=-\ad(Y_0)$, $\mathcal{D}(\ad(X_i))=i\ad(X_i)$ and $\mathcal{D}(\ad(Y_i))=-i\ad(Y_i)$ then $\mathcal{D}$ is an invertible derivation of $\ad(\jk_{2p})$.

\section{The Betti numbers for a family of solvable quadratic Lie algebras}

For each $n\in\NN$, let $\g_{2n+2}$ denote the Lie algebra with basis $\{X_0,...,X_n,Y_0,...,Y_n\}$ and non-zero Lie brackets $[Y_0,X_i]=X_i$, $[Y_0,Y_i]=-Y_i$, $[X_i,Y_i]=X_0$, $1\leq i\leq n$. Then $\g$ is quadratic with invariant bilinear form $B$ given by $B(X_i,Y_i)=1$, $0\leq i\leq n$, zero otherwise.

Let $\{\alpha,\ \alpha_1,...,\ \alpha_n,\ \beta,\ \beta_1,...,\ \beta_n\}$ be the dual basis of $\{X_0,...,X_n,Y_0,...,Y_n\}$ and set $V=\spa\{\alpha_i\}$, $W=\spa\{\beta_i\}$, $1\leq i\leq n$. It is easy to check that the associated 3-form of $\g_{2n+2}$:
\[
I=\beta\wedge\sum_{i=1}^n\alpha_i\wedge\beta_i
.\]

Denote by $\Omega_n:=\sum_{i=1}^n\alpha_i\wedge\beta_i$ then one has \[B^2(\g_{2n+2}) =\{\iota_X(I)|\ X\in\g_{2n+2}\}=\spa\left\{\beta\wedge\alpha_i,\beta\wedge\beta_i, \Omega_n\ | \ 1\leq i\leq n\right\}.\]

If $n=1$ then by we can directly calculate that $H^2(\g_4)=\{0\}$. If $n>1$, we have the non-zero super Poisson brackets:
\begin{enumerate}
	\item[(i)] $\{ I,\alpha\wedge\alpha_i\}=\alpha_i\wedge\Omega_n-\alpha\wedge\beta\wedge\alpha_i$  and $\{I,\alpha\wedge\beta_i\}=\beta_i\wedge\Omega_n+\alpha\wedge\beta\wedge\beta_i$,
\item[(ii)] $\{I,\alpha\wedge\beta\}=I$,
\item[(iii)] $\{ I,\alpha_i\wedge\alpha_j\}=2\beta\wedge\alpha_i\wedge\alpha_j$ and $\{ I,\beta_i\wedge\beta_j\}=-2\beta\wedge\beta_i\wedge\beta_j$.
\end{enumerate}

It results that $Z^2(\g_{2n+2})=\spa\left\{ \beta\wedge\alpha_i,\beta\wedge\beta_i, \alpha_i\wedge\beta_j\ |\ 1\leq i,j\leq n\right\}$ and then the second cohomology group $H^2(\g_{2n+2})= \spa\left\{ [\alpha_i\wedge\beta_j]\right\}/\spa\left\{ \left[\sum_{i=1}^n\alpha_i\wedge\beta_i\right]\right\}$, where $1\leq i,j\leq n$. So we recover the result of Medina and Revoy in \cite{MR85} obtained by describing the space $\Der_a(\g_{2n+2})$ that $b_2=n^2-1$.

To get the Betti numbers $b_k$ for $k\geq 3$, we need the following lemma.
\begin{lem}
The map $\{\Omega_n,.\}: \Wedge^k(V)\otimes \Wedge^m(W)\rightarrow \Wedge^k(V)\otimes \Wedge^m(W)$ with $k,m\geq 0$ is a vector space isomorphism if $k\neq m$ and $\{\Omega_n,\Wedge^k(V)\otimes \Wedge^k(W)\} =\{0\}$.
\end{lem}
\begin{proof} We have 
	 $\{\Omega_n,\alpha_{i_1}\wedge...\wedge\alpha_{i_k}\}=k\alpha_{i_1}\wedge...\wedge\alpha_{i_k}$, $\{ \Omega_n,\beta_{i_1}\wedge...\wedge\beta_{i_m}\}=-m\wedge\beta_{i_1}\wedge...\wedge\beta_{i_m}$ and $\{ \Omega_n,\alpha_{i_1}\wedge...\wedge\alpha_{i_k}\wedge\beta_{j_1}\wedge...\wedge\beta_{j_m}\}=(k-m)\alpha_{i_1}\wedge...\wedge\alpha_{i_k}\wedge\beta_{j_1}\wedge...\wedge\beta_{j_m}$ then the result follows.
\end{proof}

By a straightforward computation on super Poisson brackets we have the following corollary.
\begin{cor}
The restrictions of the differential $\partial$ from $\alpha\wedge\Wedge^i(V)\otimes \Wedge^j(W)$ onto $\Omega_n\wedge\Wedge^i(V)\otimes \Wedge^j(W)\oplus \alpha\wedge \beta\wedge\Wedge^i(V)\otimes \Wedge^j(W)$ and from $\Wedge^i(V)\otimes \Wedge^j(W)$ onto $\beta\wedge \Wedge^i(V)\otimes \Wedge^j(W)$ with $i,j\geq 0$, $i\neq j$	are vector space isomorphisms.
\end{cor}

Let us now give the cases for which $\ker(\partial)$ can be obtained. The following lemma is easy:
\begin{lem}\label{lem2.3}
We have $\partial\left( \Wedge^i(V)\otimes \Wedge^i(W) \right) =\partial\left(\beta\wedge \Wedge^i(V)\otimes \Wedge^j(W)\right)=\{0\}$ with $i,j\geq 0$. Moreover, $\partial\left(\alpha\wedge\beta\wedge\Wedge^i(V)\otimes \Wedge^j(W)\right)\subset \partial\left(\Wedge^{i+1}(V)\otimes \Wedge^{j+1}(W)\right)$ for all $i,\ j\geq 0$, $i\neq j$ and 
\begin{enumerate}
	\item[(i)] $\partial\left(\alpha\wedge\beta\wedge\Wedge^i(V)\otimes \Wedge^i(W)\right)=\beta\wedge\Omega_n\wedge \Wedge^i(V)\otimes \Wedge^i(W)$,
	\item[(ii)] $\partial\left(\alpha\wedge\Wedge^i(V)\otimes \Wedge^i(W)\right)=\Omega_n\wedge \Wedge^i(V)\otimes \Wedge^i(W)$.
\end{enumerate}
\end{lem}

By the reason shown in (i) and (ii) of Lemma \ref{lem2.3} we set the map \[\phi_{k_1,k_2,n}:\Wedge{}^{k_1}(\alpha_1,...,\alpha_n)\otimes\Wedge{}^{k_2}(\beta_1,...,\beta_n)\rightarrow \Wedge{}^{k_1+1}(\alpha_1,...,\alpha_n)\otimes\Wedge{}^{k_2+1}(\beta_1,...,\beta_n)\] defined by $\phi_{k_1,k_2,n} (\omega)=\Omega_n\wedge\omega$ then we have the following result.
\begin{prop}\hfill

\begin{enumerate}
	\item[(i)] If $k$ is even then
	\[\dim\ker(\partial_k)=\begin{pmatrix} n \\ \frac{k}{2}\end{pmatrix}\begin{pmatrix} n \\ \frac{k}{2}\end{pmatrix} + \sum_{i=0}^{k-1} \begin{pmatrix} n \\ i\end{pmatrix}\begin{pmatrix} n+1 \\ k-1-i\end{pmatrix} + \dim\ker\phi_{\frac{k-2}{2},\frac{k-2}{2},n}
	-\begin{pmatrix} n \\ \frac{k-2}{2}\end{pmatrix}\begin{pmatrix} n \\ \frac{k-2}{2}\end{pmatrix}.\]
	
	\item[(ii)] If $k$ is odd then
	\[\dim\ker(\partial_k)=\dim\ker\phi_{\frac{k-1}{2},\frac{k-1}{2},n} + \sum_{i=0}^{k-1} \begin{pmatrix} n \\ i\end{pmatrix}\begin{pmatrix} n+1 \\ k-1-i\end{pmatrix}.\]
\end{enumerate}

\end{prop}

Using the formula $b_{k}(\g_{2n+2}) = \dim\ker(\partial_k) + \dim\ker(\partial_{k-1}) - \begin{pmatrix} 2n+2 \\ k-1\end{pmatrix}$, the binomial identity
\[\begin{pmatrix} n \\ k-1\end{pmatrix} + \begin{pmatrix} n \\ k\end{pmatrix}=\begin{pmatrix} n+1 \\ k\end{pmatrix}\]
and the formula
\[\sum_{i=0}^{k} \begin{pmatrix} n \\ i\end{pmatrix}\begin{pmatrix} n \\ k-i\end{pmatrix}=\begin{pmatrix} 2n \\ k\end{pmatrix}\]
we obtain the following corollary.

\begin{cor}\label{cor2.5}
The $k^{th}$ Betti numbers of $\g_{2n+2}$ are given as follows:
	
	\begin{enumerate}
		\item[(i)] If $k$ is even then \[b_{k}(\g_{2n+2})= \begin{pmatrix} n \\ \frac{k}{2}\end{pmatrix}\begin{pmatrix} n \\ \frac{k}{2}\end{pmatrix} + 2\dim\ker\phi_{\frac{k-2}{2},\frac{k-2}{2},n}
				-\begin{pmatrix} n \\ \frac{k-2}{2}\end{pmatrix}\begin{pmatrix} n \\ \frac{k-2}{2}\end{pmatrix}.\]
		\item[(ii)] If $k$ is odd then \[b_{k}(\g_{2n+2})= \begin{pmatrix} n \\ \frac{k-1}{2}\end{pmatrix}\begin{pmatrix} n \\ \frac{k-1}{2}\end{pmatrix} + \dim\ker\phi_{\frac{k-1}{2},\frac{k-1}{2},n}
				+\dim\ker\phi_{\frac{k-3}{2},\frac{k-3}{2},n}-\begin{pmatrix} n \\ \frac{k-3}{2}\end{pmatrix}\begin{pmatrix} n \\ \frac{k-3}{2}\end{pmatrix}.\]
 \end{enumerate}
\end{cor}

Hence, it remains to compute $\dim\ker\left(\phi_{k,k,n}\right)$. Consider the power $\phi^m_{k_1,k_2,n}$ of the map $\phi_{k_1,k_2,n}$ and let \[K(m,k_1,k_2,n)=\dim\ker\left(\phi^m_{k_1,k_2,n}\right)\] then one has:
\begin{lem}\hfill
\begin{enumerate}
	\item[(i)] The map \[\theta^m_{k_1,k_2,n+1}:\ker\left(\phi^{m+1}_{k_1-1,k_2-1, n}\right) \oplus \ker\left(\phi^m_{k_1-1,k_2,n}\right) \oplus \ker\left(\phi^m_{k_1,k_2-1,n}\right) \]\[\oplus \ker\left(\phi^{m-1}_{k_1,k_2,n}\right)\rightarrow \ker\left(\phi^m_{k_1,k_2,n+1}\right)\]
	defined by
	\[\theta^m_{k_1,k_2,n+1} (\omega_1,\omega_2,\omega_3,\omega_4) = \alpha_{n+1}\wedge\beta_{n+1}\wedge \omega_1+\alpha_{n+1}\wedge \omega_2 + \beta_{n+1}\wedge \omega_3 \]\[+\omega_4-\frac{1}{m}\phi_{k_1-1,k_2-1,n}(\omega_1)\]
	is a vector space isomorphism.
	\item[(ii)] $K(m,k_1,k_2,n)=K(m+1,k_1-1,k_2-1,n-1) +K(m,k_1-1,k_2,n-1)+K(m,k_1,k_2-1,n-1)+K(m-1,k_1,k_2,n-1)$.
	\end{enumerate}
\end{lem}
\begin{proof}\hfill

\begin{enumerate}
	\item[(i)] The map $\theta^m_{k_1,k_2,n+1}$ is clearly injective. To prove $\theta^m_{k_1,k_2,n+1}$ surjective, let us consider $\omega\in\Wedge^{k_1}(\alpha_1,...,\alpha_{n+1})\otimes\Wedge^{k_2}(\beta_1,...,\beta_{n+1})$ such that $\Omega^m_{n+1}\wedge\omega=0$. Observe that $\omega$ can be written in the form $\omega=\alpha_{n+1}\wedge\beta_{n+1}\wedge\omega_1+\alpha_{n+1}\wedge\omega_2+\beta_{n+1}\wedge\omega_3+\omega_4$ where $\omega_1\in\Wedge^{k_1-1}(\alpha_1,...,\alpha_n)\otimes\Wedge^{k_2-1}(\beta_1,...,\beta_n)$, $\omega_2\in\Wedge^{k_1-1}(\alpha_1,...,\alpha_n)\otimes\Wedge^{k_2}(\beta_1,...,\beta_n)$, $\omega_3\in\Wedge^{k_1}(\alpha_1,...,\alpha_n)\otimes\Wedge^{k_2-1}(\beta_1,...,\beta_n)$ and $\omega_4\in\Wedge^{k_1}(\alpha_1,...,\alpha_n)\otimes\Wedge^{k_2}(\beta_1,...,\beta_n)$. By $\Omega^m_{n+1}\wedge\omega=0$, we obtain $\Omega^m_n\wedge\omega_2 = \Omega^m_n\wedge\omega_3 = \Omega^m_n\wedge\omega_4 = 0$, $\Omega^m_n\wedge\omega_1 = - m\Omega^{m-1}_n\wedge\omega_4$. It implies $\Omega^{m+1}_n\wedge\omega_1=0$ and then $\omega_1\in\ker\left(\phi^{m+1}_{k_1-1,k_2-1,n}\right)$. Moreover, $\Omega_n\wedge\omega_1 + m\omega_4\in\ker\left(\phi^{m-1}_{k_1,k_2,n}\right)$ means 
	\[\theta^m_{k_1,k_2,n+1} \left(\omega_1,\omega_2,\omega_3,\omega_4+\frac{1}{m}\phi_{k_1-1,k_2-1,n}(\omega_1)\right)=\omega.\]
	\item[(ii)] The assertion (2) follows (1).
\end{enumerate}
\end{proof}

To calculate $K(m,k_1,k_2,n)$, we use the following boundary conditions from the definition of $\phi^m_{k_1,k_2,n}$ in which we assume $\phi^0_{k_1,k_2,n}$ is the identity map:
\begin{enumerate}
	\item $K(0,k_1,k_2,n) = 0$ for all $k_1,\ k_2,\ n\geq 0$ .
	\item $K(m,0,0,n) = \begin{cases} 0, & \mbox{if } m\leq n, \\ 1, & \mbox{if } m > n.\end{cases}$
	\item $K(m,0,1,n) = K(m,1,0,n)=\begin{cases} 0, & \mbox{if } m=0 \mbox{ or } n > m, \\ n, & \mbox{if } 1\leq n\leq m.\end{cases}$
	\item $K(m,k_1,k_2,0) = \begin{cases} 1, & \mbox{if } m\geq 1, k_1=k_2=0, \\ 0, & \mbox{otherwise}.\end{cases}$
\end{enumerate}

By the condition (2) we extend $K(m,k_1,k_2,n)=0$ for negative $k_1$ or $k_2$ and by the condition (1) we set the condition (5) by $K(-m,k_1,k_2,n)=-K(m,k_1-m,k_2-m,n)$.

\begin{lem}\hfill

\[K(m,k,k,n)=\sum_{p=0}^{n} \sum_{q=0}^{n}\begin{pmatrix} n \\ p\end{pmatrix}\begin{pmatrix} n \\ q\end{pmatrix} K(m+n-p-q,k-n+p,k-n+q,0).\]
\end{lem}
\begin{proof}
By induction on $l$, we prove that 
\[K(m,k,k,n)=\sum_{p=0}^{l} \sum_{q=0}^{l}\begin{pmatrix} l \\ p\end{pmatrix}\begin{pmatrix} l \\ q\end{pmatrix} K(m+l-p-q,k-l+p,k-l+q,n-l).\]
Let $l=n$ to get the lemma.
\end{proof}

The Betti numbers of $\g_{2n+2}$ is in the case $m=1$. By the conditions (4) and (5) we reduce the following. 
\begin{cor}\hfill
\[K(1,k,k,n) = \begin{cases} 0,& \mbox{if } k < \frac{1}{2}n, 
\\ \begin{pmatrix} n \\ k\end{pmatrix}\begin{pmatrix} n \\  k\end{pmatrix}-\begin{pmatrix} n \\ k+1\end{pmatrix}\begin{pmatrix} n \\  k+1\end{pmatrix}, & \mbox{if } k \geq \frac{1}{2}n.\end{cases}\]
\end{cor}

Finally, by applying this formula we obtain the Betti number of $\g_{2n+2}$ according to Corollary \ref{cor2.5}.
\section{Appendix 1: Another way to get the Betti numbers of $\g_{2n+2}$}
In this part, we shall give another way to get the Betti numbers of $\g_{2n+2}$. It is based on the following result.
\begin{prop}\cite{Pou05}\hfill

Let $\g$ be an extension of the one-dimensional Lie algebra $\left\langle z\right\rangle$ by the Heisenberg Lie algebra $\hk_{2n+1}$, for some $n$,
\[1\longrightarrow\hk_{2n+1}\longrightarrow\g\longrightarrow\left\langle z\right\rangle\longrightarrow 0
\]
such that $\g$ acts trivially on the center $\zk=\left\langle w\right\rangle$ of $\hk_{2n+1}$. Let $\fk=\g/\zk$. Then
\[b_k(\g) =
\begin{cases}
b_k(\fk) \ \hspace{2.45cm}\text{for } k=0 \ \text{or }k=1, \\
b_k(\fk)-b_{k-2}(\fk) \ \hspace{1cm}\text{for } 2\leq k \leq n ,\\
2\left[b_{n+1}(\fk)-b_{n-1}(\fk)\right] \ \ \ \text{for } k=n+1, \\
b_{k-1}(\fk)-b_{k+1}(\fk) \ \hspace{0.6cm}\text{for } n+2\leq k \leq 2n,\\
b_{k-1}(\fk) \ \hspace{2cm}\text{for } k=2n+1 \ \text{or }k=2n+2.
\end{cases}\]
\end{prop}

It is easy to see that $\g_{2n+2}$ is an extension of the one-dimensional Lie algebra $\left\langle Y_0\right\rangle$ by $\hk_{2n+1}$. To calculate the Betti numbers of $\g_{2n+2}$ it needs to find the Betti numbers of the $2n+1$-dimensional Lie algebra $\fk$ with a basis $\{y,x_1,...,x_n,y_1,...,y_n\}$ and the Lie bracket 
\[ [y,x_i]=x_i,\ \ [y,y_i]=-y_i
\]
for all $1\leq i\leq n$.

Let $\{y^*,x_1^*,...,x_n^*,y_1^*,...,y_n^*\}$ be the dual basis of $\{y,x_1,...,x_n,y_1,...,y_n\}$.
\begin{prop}\hfill

\begin{enumerate}
	\item One has \[\partial_k\left(y^*\wedge \left(\Wedge{}^{k-1}(x_1^*,...,x_n^*,y_1^*,...,y_n^*)\right)\right)=0.
\]
\item Assume $j+l=k$ then we have
\begin{itemize}
	\item if $j=l$ then \[\partial_k\left(\Wedge{}^{j}(x_1^*,...,x_n^*)\otimes\Wedge{}^{l}(y_1^*,...,y_n^*)\right)=0,\]
	\item if $j\neq l$ then \[\partial_k\left( \Wedge{}^{j}(x_1^*,...,x_n^*)\otimes\Wedge{}^{l}(y_1^*,...,y_n^*)\right)=y^*\wedge \left(\Wedge{}^{j}(x_1^*,...,x_n^*)\otimes\Wedge{}^{l}(y_1^*,...,y_n^*)\right).\]
\end{itemize}
\end{enumerate}
\end{prop}
\begin{proof} The assertion (1) is obvious. For (2), we use the following computation:
\[\partial_k\left(x_{i_1}^*\wedge ...\wedge x_{i_j}^*\wedge y_{r_1}^*\wedge ...\wedge y_{r_l}^*\right) = (j-k) y^*\wedge x_{i_1}^*\wedge ...\wedge x_{i_j}^*\wedge y_{r_1}^*\wedge ...\wedge y_{r_l}\]
for all $1\leq i_1<...<i_j\leq n$ and $1\leq r_1<...<r_l\leq n$.
\end{proof}

It results the following corollary.
\begin{cor} The Betti numbers of $\fk$ is given as follows:
\[b_k(\fk) = \begin{pmatrix} n \\ \left[\frac{k}{2}\right]\end{pmatrix}\begin{pmatrix} n \\ \left[\frac{k}{2}\right]\end{pmatrix}
\]
where $[x]$ denotes the integer part of $x$.
\end{cor}

Applying this corollary, we have
\[b_k(\g_{2n+2}) =
\begin{cases}
1 \ \hspace{6.3cm}\text{for } k=0 \ \text{or }k=1, \\
\begin{pmatrix} n \\ \left[\frac{k}{2}\right]\end{pmatrix}\begin{pmatrix} n \\ \left[\frac{k}{2}\right]\end{pmatrix} - \begin{pmatrix} n \\ \left[\frac{k-2}{2}\right]\end{pmatrix}\begin{pmatrix} n \\ \left[\frac{k-2}{2}\right]\end{pmatrix} \ \hspace{1.1cm}\text{for } 2\leq k \leq n ,\\
2\begin{pmatrix} n \\ \left[\frac{n+1}{2}\right]\end{pmatrix}\begin{pmatrix} n \\ \left[\frac{n+1}{2}\right]\end{pmatrix} - 2\begin{pmatrix} n \\ \left[\frac{n-1}{2}\right]\end{pmatrix}\begin{pmatrix} n \\ \left[\frac{n-1}{2}\right]\end{pmatrix} \ \text{for } k=n+1, \\
\begin{pmatrix} n \\ \left[\frac{k-1}{2}\right]\end{pmatrix}\begin{pmatrix} n \\ \left[\frac{k-1}{2}\right]\end{pmatrix} - \begin{pmatrix} n \\ \left[\frac{k+1}{2}\right]\end{pmatrix}\begin{pmatrix} n \\ \left[\frac{k+1}{2}\right]\end{pmatrix} \hspace{0.55cm}\text{for } n+2\leq k \leq 2n,\\
1 \ \hspace{6.3cm}\text{for } k=2n+1 \ \text{or }k=2n+2.
\end{cases}\]
and then Theorem 2 is obtained.
\section{Appendix 2: The second cohomology group of a family of nilpotent Lie algebras}

In this appendix, in the progress of our work, we give the second cohomology of a family of nilpotent Lie algebras that are double extensions of an Abelian Lie algebra (see \cite{DPU12} for more details about these Lie algebras).

Let us denote $\g_{4n+2}$ a 2-nilpotent quadratic Lie algebra of dimension $4n+2$ spanned by $\{X,X_1,...,X_{2n},Y,Y_1,...,Y_{2n}\}$ where the Lie bracket is defined by $[Y,Y_{2i-1}]=X_{2i}$, $[Y,Y_{2i}]=-X_{2i-1}$, $[Y_{2i-1},Y_{2i}]=X$ and the bilinear form is given by $B(X,Y)=B(X_i,Y_i)=1$, zero otherwise. Let $\{\alpha,\alpha_i,\beta,\beta_i\}$ be the dual basis of $\{X,X_i,Y,Y_i\}$. We can check that the associated 3-form $I$ of $\g_{4n+2}$ is $I=\beta\wedge\Omega$ where $\Omega=\beta_1\wedge\beta_2+\beta_3\wedge\beta_4+...+\beta_{2n-1}\wedge\beta_{2n}$. Therefore, it is easy to see that $\iota_{\g_{4n+2}}(I)=\spa\{\Omega,\beta\wedge\beta_i\}$ for all $1\leq i\leq 2n$. We have the following proposition.

\begin{prop} $\dim(H^2(\g_{4n+2},\CC))=8$ if $n=1$ and $\dim(H^2(\g_{4n+2},\CC))=5n^2+n$ if $n>1$.
\end{prop}
\begin{proof}
First we need describe $\ker(\partial_2)$. Let $V$ be the space spanned by $\{\beta,\beta_1,...,\beta_{2n}\}$ then $\{I,\omega\}=0$ for all $\omega\in V\wedge V$. By a straightforward computation, we have

\begin{enumerate}
	\item $\{I,\beta\wedge\alpha_i\}=\{I,\alpha_{2i-1}\wedge \beta_{2i}\} = \{I,\alpha_{2i}\wedge \beta_{2i-1}\}=0$,
	\item $\{ I,\alpha\wedge\beta\}=I$,
	\item $\{I,\alpha\wedge\beta_{2i-1}\}=\beta_{2i-1}\wedge\Omega$, $\{I,\alpha\wedge\beta_{2i}\}=\beta_{2i}\wedge\Omega$,
	\item $\{I,\alpha\wedge\alpha_{2i-1}\}=\alpha_{2i-1}\wedge\Omega+\beta\wedge\beta_{2i}\wedge\alpha$, $\{I,\alpha\wedge\alpha_{2i}\}=\alpha_{2i}\wedge\Omega-\beta\wedge\beta_{2i-1}\wedge\alpha$,
	\item $\{I,\alpha_{2i-1}\wedge\alpha_{2j}\}=-\beta\wedge\beta_{2i}\wedge\alpha_{2j}-\beta\wedge\beta_{2j-1}\wedge\alpha_{2i-1}$, $\{I,\alpha_{2i}\wedge\alpha_{2j}\}=\beta\wedge\beta_{2i-1}\wedge\alpha_{2j}-\beta\wedge\beta_{2j-1}\wedge\alpha_{2i}$,
	\item $\{I,\alpha_{2i-1}\wedge\beta_{2j}\}=-\{I,\alpha_{2j-1}\wedge\beta_{2i}\}=-\beta\wedge\beta_{2i}\wedge\beta_{2j}$, $i\neq j$,
	\item $\{I,\alpha_{2i-1}\wedge\beta_{2j-1}\}=\{I,\alpha_{2j}\wedge\beta_{2i}\}=-\beta\wedge\beta_{2i}\wedge\beta_{2j-1}$, 
	\item $\{I,\alpha_{2i}\wedge\beta_{2j-1}\}=-\{I,\alpha_{2j}\wedge\beta_{2i-1}\}=\beta\wedge\beta_{2i-1}\wedge\beta_{2j-1}$, $i\neq j$.
\end{enumerate}
 
As a consequence, if $n=1$ then  it is direct that \[\ker(\partial_2)=V\wedge V\oplus\spa\{\beta\wedge\alpha_1,\beta\wedge\alpha_2,\alpha\wedge\beta-\alpha_1\wedge\beta_1,\alpha_1\wedge\beta_2,\alpha_1\wedge\beta_1-\alpha_2\wedge\beta_2,\alpha_2\wedge\beta_1\}.\] Therefore, we obtain $\dim(H^2(\g_{4n+2},\CC))=8$. 

In the case $n>1$ then $\Omega$ is indecomposable. Hence, \[\ker(\partial_2)=V\wedge V\oplus\spa\{\beta\wedge\alpha_{2i-1},\beta\wedge\alpha_{2i},\alpha\wedge\beta-\sum_{i=1}^n\alpha_{2i-1}\wedge\beta_{2i-1},\]\[\alpha_{2i-1}\wedge\beta_{2j}+\alpha_{2j-1}\wedge\beta_{2i},\alpha_{2i-1}\wedge\beta_{2j-1}-\alpha_{2j}\wedge\beta_{2i},\alpha_{2i}\wedge\beta_{2j-1}+\alpha_{2j}\wedge\beta_{2i-1}\}\] with $1\leq i,j\leq n$ and it is easy to check that $\dim(H^2(\g_{4n+2},\CC))=5n^2+n$.
\end{proof}


\begin{thebibliography}{9}
\bibitem{ACJ97}{G. F. Armstrong, G. Cairns and B. Jessup, Explicit Betti numbers for a family of nilpotent Lie algebras,  Proc. Amer. Math. Soc.} {\bf 125} (1997), 381-385.
\bibitem{BBM07}{I. Bajo, S. Benayadi and A. Medina, Symplectic structures on quadratic Lie algebras, J. Algebra} {\bf 316} (2007), 174-188.
\bibitem{DPU12}{M. T. Duong, G. Pinczon and R. Ushirobira, A new invariant of quadratic Lie algebras, Alg. Rep. Theory} {\bf 15} (2012), 1163-1203.
\bibitem{MR85}{A. Medina and P. Revoy, Alg\`ebres de Lie et produit scalaire invariant, Ann. Sci. \'Ec. Norm. Sup\'er.} {\bf 4} (1985), 553-561.
\bibitem{PU07}{G. Pinczon and R. Ushirobira, New Applications of Graded Lie Algebras to Lie Algebras, Generalized Lie Algebras, and Cohomology, J. Lie Theory} {\bf 17} (2007), 633-668.
\bibitem{Pou05}{H. Pouseele, On the cohomology of extensions by a Heisenberg Lie algebra,  Bull. Austral. Math. Soc.} {\bf 71} (2005), 459-470.
\bibitem{San83}{L. J. Santharoubane, Cohomology of Heisenberg Lie algebras,  Proc. Amer. Math. Soc.} {\bf 87} (1983), 23-28.
\end{thebibliography}
\end{document}